\documentclass{amsart}

\usepackage{amssymb}
\usepackage{amsmath}
\usepackage{amsthm}
\usepackage{amscd}
\usepackage{mathrsfs}
\usepackage{graphicx}
\usepackage{fullpage}
\graphicspath{{Figures/}} 

\usepackage{cancel} 
\usepackage{enumerate} 
\usepackage{hyperref} 
\usepackage[usenames,dvipsnames]{color} 
\usepackage{polynom}
\usepackage{subfig} 

\usepackage{tikz}
\usetikzlibrary{arrows}
\usetikzlibrary{calc}

\definecolor{pgray}{RGB}{75,92,113}
\definecolor{porange}{RGB}{230,148,20}
\definecolor{pblue}{RGB}{104,140,152}

\newcommand{\colorlist}[1]{%
\ifcase#1 
Blue%
\or
ForestGreen%
\or
Fuchsia%
\or
Turquoise%
\or
BurntOrange%
\or
WildStrawberry%
\or
LimeGreen%
\else
Orchid%
\fi
}

\hypersetup{colorlinks=true, linkcolor=pblue}

\theoremstyle{plain} 
\newtheorem{theorem}{Theorem}[section]
\newtheorem{proposition}[theorem]{Proposition}
\newtheorem{lemma}[theorem]{Lemma}
\newtheorem{Claim}[theorem]{Claim}
\newtheorem{corollary}[theorem]{Corollary}

\theoremstyle{definition} 
\newtheorem{definition}[theorem]{Definition}
\newtheorem{remark}[theorem]{Remark}

\newcounter{suggestcount}
\newcounter{commentlabel}
\newlength{\poincarecommentlift}
\makeatletter
\DeclareRobustCommand{\COMMENT}[1]{\@bsphack%
	\stepcounter{commentlabel}%
	\vbox to0pt{%
		\setlength{\fboxrule}{0.75pt}%
		\setlength{\fboxsep}{0.75pt}%
		\setlength{\poincarecommentlift}{1ex}%
		\addtolength{\poincarecommentlift}{\fboxrule}%
		\addtolength{\poincarecommentlift}{\fboxsep}%
		\vss\color{red}%
		\rlap{\rlap{\vrule height\poincarecommentlift width\fboxrule}\raise \poincarecommentlift%
		\hbox{\fcolorbox{red}{yellow}{%
\normalfont\footnotesize\ttfamily\bfseries\thecommentlabel}}}}%
	\marginpar{\noindent\raggedright%
	\textbf{\color{red}\thecommentlabel}:\thinspace\footnotesize#1}%
}
\title{ Brill-Noether loci with ramification at two points.}
\author{montserrat Teixidor i Bigas}
\address{Mathematics Department, Tufts University, 503 Boston Avenue, Medford MA 02155, USA}

\let\oldtocsection=\tocsection
\let\oldtocsubsection=\tocsubsection
\renewcommand{\tocsection}[2]{\hspace{0em}\oldtocsection{#1}{#2}}
\renewcommand{\tocsubsection}[2]{\hspace{1.1em}\oldtocsubsection{#1}{#2}}

\begin{document} 

\maketitle


\setlength{\parindent}{0cm} 
\setlength{\parskip}{\baselineskip} 
\setlength{\abovedisplayskip}{0.7\baselineskip} 
\setlength{\belowdisplayskip}{0.7\baselineskip}

\begin{abstract} We prove the injectivity of the Petri map for linear series  on a general curve with given ramification at two generic points. 
We also describe the components of such a set of linear series on a chain of elliptic curves.

\end{abstract}

Brill-Noether loci $G_d^r$  parametrize linear series $(L, V)$ where $L$ is a line bundle of degree $d$ and $V$ an $r+1$ dimensional space of sections of $L$ .
The behavior of these Brill-Noether loci is governed by the Petri map defined as the natural cup product $P: V\otimes H^0(C, K\otimes L^{-1})\to H^0(C,K)$. 
Its image is the orthogonal to the tangent space to $G_d^r$ at $(L, V)$. 
It was proved by Gieseker (\cite{G}) that this map is injective  for every linear series on a generic curve.

In order to understand Brill-Noether loci for generic curves,
 it is often useful to introduce  degeneration techniques  that reduce questions about non-singular curves to questions about singular, usually nodal and often reducible curves.
 One of the most well behaved  degenerations are chains of elliptic and rational curves introduced by Welters in \cite{W} and used profusely since then. 
This type of degeneration requires  considering a  generalized Brill-Noether  locus consisting of  linear series with given ramification at two fixed points:
Given a curve $C$ of genus $g$, integers $r, d, \alpha=(\alpha_0,\dots, \alpha_r),\beta=( \beta_0, \dots, \beta_r) $ with   $0\le \alpha_0\le \dots\le \alpha_r,\ \beta_0\ge \dots\ge \beta_r\ge 0$, 
we consider the space $G^r_d(C,P,Q,\alpha, \beta)$ of linear series of degree $d$ and dimension $r$ on $C$ with ramification at least $\alpha$ at $P$ and $\beta$ at $Q$.

There is a corresponding Petri map (see section \ref{sec:tansp} )  whose injectivity for a given line bundle and space of sections is equivalent to the non-singularity of this scheme at the given point.
This map was considered in the case of a single point in \cite{CHT}.
Here we prove the injectivity of the map for  two points for generic curve and exact ramification at the points. This result has also been recently obtained by Pflueger (see \cite{P}).
 From standard techniques, it is sufficient to prove the result for limit linear series on reducible curves.
 The original proof of Gieseker, as well as easier proofs provided later \cite{EH1}, \cite{CT} or tropical proofs (\cite {JP}), all in the case of no ramification points  show that the kernel of the Petri map is zero.
Generalizations to the case of imposing ramification at a single point cited above or considering higher rank vector bundles  (see for instance \cite{T},  \cite{CLT}), as well as  Pflueger work,
  follow the same pattern.
 Here, we use  the methods of \cite{LOTZ1}, \cite{LOTZ2}. Instead of proving that the kernel of the Petri map is zero, we show that the image is large enough.

A study of the loci  $G^r_d(C,P,Q,\alpha, \beta)$ bypassing the Petri map and focusing instead on the structure of the corresponding loci on elliptic curves is provided in \cite{COP}.

We also include in this work the  explicit description of the space of limit linear series on a chain of elliptic curves.
We show the correspondence between components of the moduli space of linear series with fixed ramification on a chain of elliptic curves and fillings of a skew shape generalizing Young Tableaux.

The first appearance of Young Tableaux to parameterize linear series on a reducible curve with $g$ elliptic components seems due to  Edidin (see \cite{E}).
He  introduced them in the case when the Brill-Noether number is 0  and therefore  there is only a finite number of linear series on the curve. 
They were used again in the  case $\rho =0$ in the tropical context in \cite{CDPR}.
A generalization was considered in \cite{LT}, for arbitrary $\rho$ and describing $W^r_d$, that is the image of $G^r_d$ in the Jacobian.
In \cite{CLPT}, we built a  description $G^r_d(C,P,Q,\alpha, \beta)$ tailored only to the case of $\rho =1$.

\section{Preliminaries on limit linear series.}\label{LLS}

Recall that a connected nodal curve is stable if  every rational component has at least three nodes. 
We will also consider some semistable curves with rational  components with only two nodes.
A  semistable curves is of compact type if its dual graph is a tree or equivalently its jacobian is compact.
In this paper, we will use {\bf generic chains of elliptic and rational curves}: 
given elliptic curves $E_1, \dots E_g$ with generic  marked points $P_i,Q_i\in E_i$, glue $Q_i$ to $P_{i+1}, i=1,\dots, g-1$ and insert  chains of rational curves at nodes and perhaps at $P_1, Q_g$.
We will number the components $C^1,\dots, C^N$ and the points giving rise to the nodes  with $P^i, Q^i\in C^i$.

We start by reviewing the concept of limit linear series introduced by Eisenbud and Harris and its generalizations by Osserman to linked linear series.

\begin{definition}\label{def:lls} (see \cite{EH2})  Let $C$ be a curve of compact type with components $C^i, i=1,\dots, N$.
A {\bf  limit linear series} of degree $d$ and dimension $r$ on a $C$  is given as 
 $(L^i, V^i), i=1,\dots, N$ with $L^i$ a line bundle of degree $d$ on $C^i$ and $V^i$ an $(r+1)$-dimensional space of sections of $L^i$
satisfying the following compatibility condition: let  the node $N$ be obtained by identifying the point $P$ on the component $C^i$  with the point $Q$ on the component $ C^j$
and  $u_0< \dots< u_r,\ v_0> \dots> v_r$ be the orders of vanishing of $V^i, V^j$ at $P, Q $ respectively, then, $u_k+v_k\ge d, k=0,\dots, r$.
The series is said to be {\bf refined }if the inequality is an equality at every node and for every vanishing pair.
 \end{definition}
 
 The concept of limit linear series appears when considering a family of linear series with general fiber non-singular and special fiber a curve of compact type.
 On the central fiber, one can concentrate all the degree of the limit line bundle on one irreducible component  which will then necessarily carry also the limit of the space of sections.
 As we are interested in limit linear series as the limit of a linear series in a non-singular curve, we can assume that the limit linear series are refined. 
If they are not, a few blow ups in the family will take us to this situation.
It is also easy to add assigned vanishing at other given points(away from the nodes). 
In this paper we will consider two additional general points that we will place at the beginning and end of the curve.

 An alternative to limit linear series that concentrates the degree successively on one component
  is dividing the degree of the limit vector bundle (and therefore also its spaces of sections) arbitrarily among the components of the limit curve of compact type.
 On a family with general fiber irreducible and central fiber of compact type, one can go from one multi-degree on the central fiber 
 to a different one by tensoring with a line bundle with support on the central fiber.
 This process gives rise to  natural maps between the different line bundles so obtained that vanishes on some components and induces maps among the spaces of sections. 
 
\begin{definition}\label{def:llinks} (see \cite{O1}, \cite{O2}) Let $C$ be a curve of compact type with components $C^i, i=1,\dots, N$.
A {\bf linked linear series} of degree $d$ and dimension $r$ on a $C$  is given as a line bundle ${\mathcal L}$ on $C$ and for any multi-degree on $C$ of total degree $d$
a space of sections of dimension $r+1$  compatible with the natural maps among the line bundles obtained from ${\mathcal L}$ modifying the degree to the given one.
 \end{definition}
 
From any linked linear series, one can obtain a limit linear series by ignoring the additional data.
Vanishing conditions at the nodes are a consequence of the compatibility conditions.
 One advantage of the linked linear series approach is that it allows to relate sections on different irreducible components to each other and think of them as global sections.

The scheme of refined limit linear series of degree $d$ and dimension $r$ 
is a union over the set of possible vanishing at the nodes of a product of spaces of linear series on the components with assigned vanishing. 

\begin{definition}\label{def:eh-struct}
Let $C$ be a generic chain of  elliptic and rational curves of total genus $g$. Choose $ r, d, \alpha =(\alpha_0,\dots, \alpha_r), \beta =(\beta_0,\dots, \beta_r), $ non-negative integers.
with $\alpha_0\le \dots\le \alpha_r,\ \beta_0\ge \dots\ge \beta_r$.
Consider all possible choices of $\alpha(i), \beta(i)$ such that
\[ \alpha (1)=(\alpha_0,\dots, \alpha_r), \beta (N)=(\beta_0,\dots, \beta_r), 
 \alpha _j(i)+\beta _j(i-1)=d-r, j=0,\dots, r, i=2,\dots, N\]
 still satisfying $\alpha_0(i)\le \dots\le \alpha_r(i),\ \beta_0(i)\ge \dots\ge \beta_r(i), \ i=1,\dots, N$.
 The {\bf Eisenbud-Harris scheme} $G^r_d(C,P_1,Q_N,\alpha, \beta)$ is the scheme obtained as the union
$$\bigcup _{\alpha, \beta}\prod_{i=1}^g G^{r}_d (C^i,P^i,Q^i, \alpha (i),\beta(i)) \subseteq \prod_{i=1}^g G^r_d(C^i)$$
The {\bf  Brill-Noether number} is defined as 
  \[\rho (g, d,r, \alpha, \beta)=g-(r+1)(g-d+r)-\sum_{j=0}^r \alpha_j-\sum_{j=0}^r  \beta_j\]
\end{definition}

This $\rho (g, d,r, \alpha, \beta)$ is the expected dimension of  the Brill-Noether  locus of linear series of degree $d$ and dimension  $r$
with ramification orders at two fixed  points given as $\alpha_0,\dots, \alpha_r, \beta_0,\dots, \beta_r$.
In what follows, we will  also use the vanishing (rather than ramification) orders corresponding to ramifications $\alpha, \beta$ defined as 
 \[ u(i)=(u_0(1),u_1(i), \dots,u_r(i))=(\alpha_0(i), \alpha_1(i)+1,\dots, \alpha_r(i)+r), \]
\[v(i)=(v_0(i),v_1(i), \dots,v_r(g))=(\beta_0(i)+r, \beta_1(i)+r-1,\dots, \beta_r(i))\]

\begin{lemma} \label{lem:van2psc} Let $C$ be a non-singular curve, $P,Q\in C$ arbitrary points, $L$ a line bundle on $C$, $V\subseteq H^0(C,L)$ an $r+1$ dimensional space of sections,
 $u_0<\dots <u_r, \ v_0>\dots >v_r$  the distinct orders of vanishing of the sections in $V$ at the points $P,Q$.
Then, there exist a basis $s_0,\dots , s_r$ and a permutation $\sigma$ of $0,\dots, r$ such that 
$ord _Ps_j=u_j, ord _Qs_j=v_{\sigma (j)}$.
\end{lemma}
\begin{proof}
Choose a basis $s_j$ such that $ord _Ps_j=u_j$. If the orders of vanishing at $Q$ of these sections are all different, then this basis works.
Otherwise, choose $v_k$  with $k$ as large as possible (so $v_k$ as small as possible) such that there are several sections 
 $s_{i_1}\dots s_{i_t} ,\ i_1<\dots < i_t$ that have this order of vanishing $v_k$ at $Q$. 
 Replace  $s_{i_1}\dots s_{i_{t-1}} $ by  $s_{i_1}+\lambda_{i_1}s_{i_t}\dots s_{i_{t-1}}+\lambda_{i_{t-1}}s_{i_t} $,
  we obtain a new basis with the same orders of vanishing at $P$ but higher order of vanishing at $Q$.
  Repeat the process if needed.
\end{proof}

We will say that the linear series has a {\bf swap} on $C $ if $\sigma\not= Id$ in the lemma above.

 \begin{lemma} \label{lem:van2pchain} Let $C$ be a chain of elliptic and rational curves curve, $P^1\in C^1,Q^N\in C^N$.
 Let $(L^i, V^i)_{i=1,\dots, N}\in G^r_d(C,P^1,Q^N,\alpha, \beta)$ be an exact  limit linear series on the chain.
 One can then choose basis $s_0^i, \dots s_r^i$ of $V^i$ such that $ord_{P^1}s_j=\alpha _j+j$, $ord_{Q^N}s_j^N=\beta _{\sigma (j)}+r-\sigma (j)$
 and for a fixed $i$,  the orders of vanishing of the sections $s_j^i$ at $P_i, Q_i$ are all different (and therefore are the orders of vanishing of the limit linear series)
 and $ord_{Q^i}s_j^i+ord_{P^{i+1}}s_j^{i+1}=d$.
\end{lemma}
\begin{proof}
Proceed by induction on $i$ using Lemma \ref{lem:van2psc}.
Note that here the order of vaninshing of the sections $s_0^i,\dots s_r^i$ at $P^i$ are not necessarily in ascending  order except at $P^1$ 
and the orders of vanishing at $Q^i$ are not necessarily in descending order at any $Q^i$. 
\end{proof}

 \begin{definition} \label{def:section} Let $C$ be a chain of elliptic and rational curves curve, with components $C^i, i=1,\dots, N$.
 Let $(L^i, V^i)_{i=1,\dots, N}\in G^r_d(C,P^1,Q^N,\alpha, \beta)$ be a limit linear series on the chain.
 We will say that  $(s^i) \ i=1,\dots, N$ is {\bf a section of the series} if  $s^i\in V^i$ 
 and $ord_{Q^i}s^i+ord_{P^{i+1}}s^{i+1}\ge d$.
\end{definition}

\begin{lemma}\label{prop:K=O(aP+bQ)} Let $(s^i)i=1,\dots, N$ be a limit section of the canonical linear series on a chain of elliptic and rational curves. 
There exists at least one  value of $i$ such that $C^i$ is an elliptic curve and the section $s^i$ satisfies 
$ord_{P^i}s^i+ord_{Q^i}s^i=2g-2$.
Moreover, if $\bar i$ is the index of the last elliptic component at which this is true and for any $i$ we write $ l^i= \text{ number of elliptic curves preceding } C^i$, then if  $i> \bar i$, 
$ord_{P^i}s^i<2l^i$
\end{lemma}
\begin{proof}
The canonical limit linear series on a chain of rational and elliptic curves corresponds to 
\[ (M^i, W^i), \ M^i=\mathcal {O}_{C^i}(2l^iP^i+2(g-l^i-1)Q^i), \ l^i= \text{ number of elliptic curves preceding } C^i.\]
 Write $u^i=ord_{P^i}s^i, v^i=ord_{Q^i}s^i$.  For ease of notation, we write $u^{N+1}=2g-2-v^N$.
 
From the condition of limit linear series, $v^i+u^{i+1}\ge d=2g-2$.
From the fact that $s^i$ is a section of a line bundle of degree $2g-2$, $u^i+v^i\le 2g-2$.
Therefore, $u^{i+1}\ge u^i$ and equality implies that $u^i+v^{i}= d=2g-2$.

As $u^1\ge 0, u^{N+1}\le 2g-2$, there exists an   elliptic component $C^{i_0}$ such that  $u^{i_0}\ge 2l^{i_0}$ and for the next elliptic component in the chain $C^{i_1}$,
$ u^{i_1}< 2l^{i_1}=2l^{i_0}+2$ or equivalently, $ u^{i_1}\le 2l^{i_0}+1$.
Then,  also $ u^{i_0+1}\le u^{i_1}\le 2l^{i_0}+1$. It follows that $v^{i_0}\ge 2g-2-u^{i_0+1}\ge 2g-2l^{i_0}-3$.
As $s^{i_0}$ is a section of the line bundle $\mathcal {O}_{C^i}(2l^{i_0}P^{i_0}+2(g-l^{i_0}-1)Q^{i_0})$, the conditions  $u^{i_0}\ge 2l^{i_0}, v^{i_0}\ge 2g-2l^{i_0}-3$ implies $v^{i_0}= 2g-2l^{i_0}-2$ and the first claim is proved.

To prove the second claim, assume that $C^{\bar i}$ is the last elliptic component for which $ord_{P^{\bar i}}s^{\bar i}+ord_{Q^{\bar i}}s^{\bar i}=2g-2$ and that there is a
curve $C^{\hat i}$ further along for which $ord_{P^{\hat i}}s^{\hat i}\ge 2l^{\hat i}$. 
Then, the argument above starting at $\hat i$ shows that there exist a value of i larger than $\hat i$ for which the curve is elliptic and the sum of vanishing adds up to $2g-2$,
contradicting the choice of $\bar i$.
\end{proof}

\begin{proposition} \label{prop:indepsecK} Consider a generic chain of elliptic and rational curves. Let $(t_0^i), \dots (t_{\bar r}^i)$ be limit sections of the canonical linear series  such that 
the values of $i$ such that  $ ord_{P^i}t_j^i+ord_{Q^{i}}t_j^{i}=2g-2$ are all distinct.
Then the $ (t_j^i), j=1,\dots, \bar r$ are limits of linearly independent sections of the canonical  series.
\end{proposition}
\begin{proof}
Assume that we had a dependence between the $t_j$. 
Choose the index $i$ of the last elliptic component on which for one of the values of $j$, $ ord_{P^i}t_j^i+ord_{Q^{i}}t_j^{i}=2g-2$.
By assumption, this value of $j$ is unique.
Consider the aspects of all the sections on this component $C^i$.
The canonical limit linear series  on the component $C^i$ has line bundle given as $ M_i=\mathcal {O}_{C_i}(2l^iP_i+2(g-l^i-1)Q_i)$.
Therefore, the order of vanishing $ord_{Q^i}s_j^i=2g-2-2l^i$.
From Proposition \ref {prop:K=O(aP+bQ)}, $ ord_{P^{i+1}}s_{j'}^{i+1}<2l^{i+1}=2l_i+2$ for all $j'\not=j$.
Therefore, $ ord_{Q^{i}}s_{j'}^{i}\ge 2g-2l^{i}-3$
Therefore, $s^i_j$ has the smallest order of vanishing at $Q^i$ of all the sections.
It follows that in the linear dependence condition among the sections the coefficient of $s_j$ is 0.
Repeat then the process with the remaining sections and the resulting dependence.
From Proposition \ref{prop:K=O(aP+bQ)},  for each $j, j=1,\dots, \bar r$, there is at least one component $C^i$ at which $ ord_{P^i}t_j^i+ord_{Q^{i}}t_j^{i}=2g-2$.
Therefore, all the coefficients in the linear combination are zero.
\end{proof}

\section{Tangent space}\label{sec:tansp}

In  \cite{CHT}, we studied the Brill Noether locus with fixed vanishing at a single point.  (that is  $G^r_d(C,P,Q,\alpha, \beta)$ with $\beta$ being the 0 sequence).
 Assume first the curve $C$ to be non-singular and let  $(L,V)\in G^r_d(C,P,Q,\alpha, 0)$
 We described  the tangent space as follows: 
There is a natural map 
\[ \mu_r: H^1(C, {\mathcal O}_C)\to  Hom(V(-(\alpha_r+r)P), H^1(C, L(-(\alpha_r+r)P))\]
One can define inductively 
\[ \mu_j: Ker \mu_{j+1}\to  Hom(V(-(\alpha_j+j)P)/ V(-(\alpha_{j+1}+(j+1))P), H^1(L(-(\alpha_j+j)P))\]
 then the tangent space to the Brill-Noether locus at a complete linear series whose ramification sequence at $P$  is precisely $\alpha$  is the orthogonal to the image of the map $\mu_0$
  (see  Lemma 3.1 in  \cite{CHT}).

Choose a basis $s_0,\dots ,s_r,$ of $V$ with $s_j\in V(-(\alpha_j+j)P)- V(-(\alpha_{j+1}+(j+1))P)$.
The natural cup-product maps 
\[V(-(\alpha_j+j)P)\otimes H^0(C, K\otimes L^{-1}((\alpha _j+j)P))\to  H^0(C, K)\]
give rise to a map 
\[ P: \oplus_{i=0}^r[<s_j>\otimes H^0(C, K\otimes L^{-1}(\alpha _j+j)P))]\to  H^0(C, K)\]
By duality, the orthogonal to the image of this generalized Petri map is the tangent space to the Brill-Noether locus with given ramification at the point $P$.
For generic curve and generic choice of point $P$, the Petri map is know to be injective for a complete linear series with precisely the required ramification(see  Proposition 3.2 in  \cite{CHT}).
This shows that the Brill-Noether locus with one ramification point is non-singular of the expected dimension on the generic pointed curve.

Our goal is to extend this result to the case of two marked points and given ramification.
Given a second point $Q$, orders of ramification $(\beta _0, \dots ,\beta _r)$ and $(L,V)\in G^r_d(C,P,Q,\alpha, \beta)$,
denote by $\mu_{0, P,\alpha} , \mu_{0, Q,\beta} $ the two maps constructed as above from the two sets of data.
Then the tangent space to the locus of linear series with given vanishing at $P, Q$ is the orthogonal to the sum of the images of the two maps.

Given a complete linear series with given ramification and using Lemma \ref{lem:van2psc}, we can  choose a basis $s_0,\dots ,s_r, $ of $V$
 whose orders of vanishing at $P, Q$ are the orders of vanishing of the linear series at these two points.
In particular, 
\[  \ s_j\in V(-(\alpha_j+j)P-(\beta_{\sigma(j)}+r-\sigma(j))Q).\]
We have a generalized Petri map 
\begin{equation}\label{eq:Petrmap} \mu_{P,Q,\alpha, \beta}: \oplus_{j=0}^r[<s_j>\otimes H^0(C, K\otimes L^{-1}(\alpha _j+j)P)+(\beta_{\sigma(j)}+r-\sigma(j))Q)]\to  H^0(C, K)\end{equation}
The orthogonal to the image of $\mu_{P,Q,\alpha, \beta}$ is the tangent space to the Brill-Noether locus $G^r_d(C,P,Q,\alpha, \beta)$ with ramification at the points $P, Q$.
We want to show that this map is injective for generic  $C,P, Q$ 

\begin{proposition} Consider $g,r,d,\alpha=(\alpha_0,\dots \alpha_r), \beta=(\beta_0,\dots,\beta_r)$ non-negative integers satisfying 
\[ g-d+r\ge 0, \alpha_0\le \dots\le \alpha_r, \beta_0\ge \dots\ge \beta_r,  g-(r+1)(g-d+r)-\sum \alpha_j-\sum \beta_j\ge 0.\]
Let $(L, V)\in G^r_d(C,P,Q,\alpha, \beta)$ be a linear series with ramification at $P, Q$ precisely $\alpha, \beta$.
Then the Petri map for $(L,V)$ is injective.
\end{proposition}
\begin{proof} Rather than proving that the kernel of $\mu_{P,Q,\alpha, \beta}$ is 0, we will show that the image has dimension at least 
$ (r+1)(g-d+r)+\sum \alpha_j+\sum \beta_j$.

We can degenerate $C$ to a generic chain of elliptic and rational curves and prove the result instead for limit linear series on such a curve.
Up to some blow ups and desingularizations, we can assume that the linear series is refined.

Choose   basis $(s_0^i), \dots ,(s_r^i)$ of the original linear series as in Lemma \ref{lem:van2pchain}.
By assumption, the order of vanishing of $s_k^1$ at $P^1$ is $\alpha_k+k$.
From the condition that $s_k^1$ is a section of a line bundle of degree $d$ on $C^1$, $ord_{Q^1}s_k^1\le d-\alpha_k-k$.
Then, from the choice in our basis  $ord_{P^2}s_k^2\ge \alpha_k+k$.
With the same argument, the inequality holds in fact at every $P^i$.
Similarly, the vanishing at each $Q^i$ is at least $\beta_{\sigma(k)}+r-\sigma(k)$.
This allows us to substitute $L^i$ with   $L^{'i}=L^i(-(\alpha _k+k)P^i+(\beta_{\sigma(k)}+r-\sigma(k))Q^i)$ and $(s_k^i), i=1,\dots, N$ by the corresponding sections $(s_k^{'i})$ of $L'_i$.
From $ ord_{P^{i+1}}s_k^{i+1}+ord_{Q^i}s_k^i=d$, we obtain 
\[ ord_{P^{i+1}}s_k^{'i+1}+ord_{Q^i}s_k^{'i}=d-\alpha_k-\beta_{\sigma (k)}-r-k+\sigma (k),\]
giving rise to a limit linear series associated to $L(-(\alpha_k+k)P^1-(\beta_{\sigma(k)}+r-\sigma(k))Q^N)$ (of dimension zero, as we have a single section).

For each $k$, choose basis of the  Serre dual linear series associated to  $K\otimes L^*((\alpha_k+k)P^1+(\beta_{\sigma (k)}+r-\sigma(k))Q^N)$
  \[  ( \bar s_0^{i,k}), \dots , ( \bar s_{r_k}^{i,k}), \ i=1,\dots, N, k=0,\dots , r , \ \ r_k\ge g-d+r+\alpha_k+\beta_{\sigma (k)}+k-\sigma(k)  \]
satisfying $ord_{P_i}\bar s_j^{i,k}+ord_{Q_i}\bar s_j^{i,k}=2g-2-d+\alpha_k+k+\beta_{\sigma(k)}+r-\sigma (k)$.
From the compatibility conditions for the vanishing at the nodes $Q^i, P^{i+1}$ for both  $(s_k^{'i}) ( \bar s_j^{i,k})$, we obtain compatibility conditions at the nodes  for their product  $s_k^{'i}\bar s_0^{i,k}$.
By Proposition \ref{prop:K=O(aP+bQ)},  for each $k=1,\dots, r, j=1,\dots , r_k$ 
there is at least one elliptic component $C^i$ at which the sums of the vanishing orders at $P^i, Q^i $of  $s_k ^{'i} \bar s_j^{i,k}$ add up to $2g-2$. 
This requires  the sums of the vanishing orders at $P^i, Q^i $ of  both $s_k ^{'i}$ and $ \bar s_j^{i,k}$ add up to the corresponding degrees of the line bundles.
As the $(s_k ^{'i})$ are derived from the $(s_k ^i)$, the vanishing orders at $P^i, Q^i $ of the $s_k ^I$ at this component $C^i$ must also add to $d$.
Because the $s_k$ are part of the same linear series,  the assumption that the points $P^i,Q^i$ are general on $C^i$ means that all the values of $i$ are different.
Then, Proposition \ref{prop:indepsecK}, shows the independence claimed.
\end{proof}

\section{Components of the space of Limit linear series on chains of elliptic curves}\label{LLSdesc}

In this section, we describe the components of the moduli space of limit linear series on a chain of elliptic curves. 
We start with its basic elements, namely linear series on a single elliptic curve.

\begin{lemma} \label{exfibanul} Consider an elliptic curve $C$, two points $P, Q\in C$ such that $P-Q$ is not a torsion point, 
 and integers $0\le u_0<\dots<u_r\le d$, $d\ge v_0>\dots>v_r\ge 0$.
There exist linear series of degree $d$ and dimension $r$ with vanishing  at least  $ u_0,\dots ,u_r$ at $P$, $v_0, \dots v_r$ at $Q$ 
if and only if $u_j+v_j\le d$ and equality holds for at most one  $j\in \{0,\dots, r\} $.
Moreover, if the above conditions hold, the space of limit linear series satisfying the given ramification conditions 
is an open set in the intersection of two Schubert cycles or of a bundle of intersections of Schubert cycles over the elliptic curve depending on whether there exists a $j$ with $u_j+v_j=d$ or not. 
It has dimension 
$$\sum_{j=0,\dots, r} (d-u_j-v_j)-r=(r+1)d-r-\sum_{j=0,\dots, r} (u_j+v_j)$$
\end{lemma} 

\begin{proof} Assume that such a linear series exists. Then there is a line bundle $L$ of degree $d$ and an $r+1$-dimensional space of sections $V$ satisfying the vanishing conditions.
In particular, $\dim V(-u_jP)\ge j+1, \dim V(-v_jQ)\ge r+1-j$.
 As these two vector spaces live inside a vector space of dimension $r+1$, they intersect. 
 Therefore, there is a section that vanishes at $P$ with order at least $u_j$ and at $Q$ with order at least $v_j$. This implies that $u_j+v_j\le d$.
 Moreover, equality implies that $L=\mathcal {O}(u_jP+v_jQ)$. As $P-Q$ is not a torsion point, this can happen for at most one value of $j$.
 So the conditions are necessary. 
 
 They are also sufficient, if $u_{j_0}+v_{j_0}= d$, take $L=\mathcal {O}(u_{j_0}P+v_{j_0}Q)$. If not such $j_0$ exists, take $L$ generic.
 Then choose a generic section $s_j\in H_0(L(-(u_jP+v_jQ)))$. 
  We claim that such a generic choice produces a set of linearly independent sections. Assume this is not the case. There exist then some constants, not all of them zero, such that 
 \[ \lambda _0s_0+\lambda_1s_1+\dots+\lambda_rs_r=0\]
 If $\lambda _0\not= 0$, then $s_0$ is in the span of $s_1,\dots s_r$ and therefore in $H^0(L(-u_1P-v_0Q))$.
 Note that $u_1>u_0$ implies that $H^0(L(-u_1P-v_0Q))\subseteq H^0(L(-u_0P-v_0Q)$.
  The section $s_0$ was chosen generically in $H^0(L(-u_0P-v_0Q)))$. Therefore,  $s_0$ vanishing at $P$ with order $u_1$ would require $H^0(L(-u_1P-v_0Q))=H^0(L(-u_0P-v_0Q))$. 
  From $u_1>u_0$, this  can only happen if $L(-u_1P-v_0Q)=\mathcal{O}, L(-u_0P-v_0Q)=\mathcal{O}(P)$. 
    But $L$ was chosen to be either generic or equal to $\mathcal {O}(u_iP+v_iQ)$.
  Then, the genericity of $P,Q$ prevent it from also being equal to $\mathcal {O}(u_1P+v_0Q)$. Therefore, $\lambda_0=0$. One can then proceed similarly to show that $\lambda_1=0$ and so on.  
  So, a generic choice of sections makes them linearly independent and therefore $r+1$-dimensional spaces of such sections exist.
  
   The sections $s_j$  are in  $L(-u_jP-v_jQ)$ which has dimension $d-u_j-v_j$ if  $L\not= \mathcal {O}(u_jP+v_jQ)$ and dimension one if equality holds.
     The  line bundle is uniquely determined if  for some $j$, $u_j+v_j=d$, otherwise, it varies in a one-dimensional family.
The choice of the line bundles and the sections up to multiplication with a constant determines the linear series.

Note that not all linear series with the given vanishing have a basis as described.
 This can fail because at least one of the orders of vanishing at $P$ (resp.  $Q$) is strictly bigger than the given $u_i$ (resp $v_i$).
 It can also fail because there is a  sections vanishing at $P$ with order $u_{j_1} $ and at $Q$ with order $v_{j_2}, {j_1}\not={j_2}$, while some other section vanishes with order $v_{j_1}$ at $Q$ and some  
 $u_{j_3}$ at $P$ where $j_3$ may or may not be equal to $j_2$ (we called this  `` swap'').
 In both situation, the linear series so obtained are in the closure of the open set described. 
 In fact, when the exceptional behavior happens, one can construct a family of subspaces of dimension $r+1$ of $H^0(L)$ whose generic member has the generic behavior 
 while some special member has  the particular behavior.
 Therefore, the set of linear series has the stated dimension.
  
   The spaces of section 
 \[ H^0(L(-u_rP))\subset H^0(L(-u_{r-1}P))\subset \dots \subset H^0(L(-u_0P)), \]
  \[  H^0(L(-v_0Q))\subset H^0(L(-v_{1}Q))\subset \dots \subset H^0(L(-v_rQ))\] 
 have dimension $d-u_r, d-u_{r-1},\dots, d-u_0, \ \ d-v_0, d-v_1,\dots, d-v_r$.
 A linear series satisfies the required conditions if 
 \[\dim (V\cap H^0(L(-u_rP)))\ge 1,  \dim (V\cap H^0(L(-u_{r-1}P)))\ge 2\dots \dim (V\cap H^0(L(-u_0P)))\ge r+1\]
  \[ \dim (V\cap H^0(L(-v_0Q)))\ge 1, \dim (V\cap H^0(L(-v_{1}Q)))\ge 2 \dots \dim (V\cap H^0(L(-v_rQ)))\ge r+1 \] 
  Therefore, the space of linear series is an open set in the intersection of the two Schubert cycles when the line bundle is determined and of a bundle 
  of intersection of Schubert cycles over the Picard variety of the curve when the line bundle is free to vary.
    \end{proof}  

\begin{figure}[h!]
\begin{tikzpicture}[scale=.6]

\draw (0,0)--(0,-5);
\draw[very thick] (0,0)--(0,-3);%
\draw (1,0)--(1,-5);
\draw (2,1)--(2,-4);
\draw (3,2)--(3,-4);
\draw (4,2)--(4,-3);
\draw[very thick](4,0)--(4,-3);
\draw[very thick] (0,0)--(4,0);
\draw (2,1)--(4,1);
\draw (3,2)--(4,2);
\draw (0,-1)--(4,-1);
\draw (0,-2)--(4,-2);
\draw[very thick] (0,-3)--(4,-3);
\draw (0,-4)--(3,-4);
\draw (0,-5)--(1,-5);

\draw[dotted](0,3)--(0,0);
\draw[dotted](1,3)--(1,0);
\draw[dotted](2,3)--(2,1);
\draw[dotted](3,3)--(3,2);
\draw[dotted](4,3)--(4,2);
\draw[dotted](0,1)--(2,1);
\draw[dotted](0,2)--(3,2);

\draw[dotted](0,-5)--(0,-6);
\draw[dotted](1,-5)--(1,-6);
\draw[dotted](2,-4)--(2,-6);
\draw[dotted](3,-4)--(3,-6);
\draw[dotted](4,-3)--(4,-6);
\draw[dotted](3,-4)--(4,-4);
\draw[dotted](1,-5)--(4,-5);

\end{tikzpicture}
\caption{ Skew shape $r=3, g-1-d+r=2, \alpha=(0,0,1,2), \beta=(2, 1, 1,0)$.}
\label{fig:skewshape}
\end{figure}
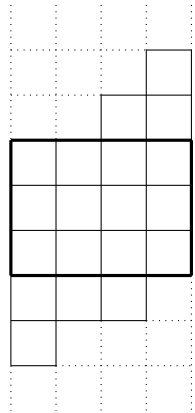
\begin{definition} \label{def:sshape} A {\bf Skew Shape}  is obtained as follows: take a rectangle with $r+1$ columns and $g-d+r$ rows.
 Number  the columns  left to right from $0$ to $r$ and the rows  top to bottom from 0 to $g-1-d+r$. 
 Add to column $j$ of this rectangle  $\alpha_j$ boxes on the top (negative) side and $\beta_j$ boxes to the bottom.
  Column $j$ now goes from position $-\alpha_j$ to $g-1-d+r +\beta_j$ and contains $g-d+r +\beta_j+\alpha_j$ boxes.
\end{definition}
 We consider the  $r+1$ columns containing the skew shape as infinite columns extending above and below the  box.
Give initial weight 1 to all the boxes  directly above the upper contour of this shape and weight zero to the rest.
Consider possible ways of filling the shape, including perhaps some boxes directly above or below the shape, with numbers between 1 and g, each with weight 1 or $-1$.
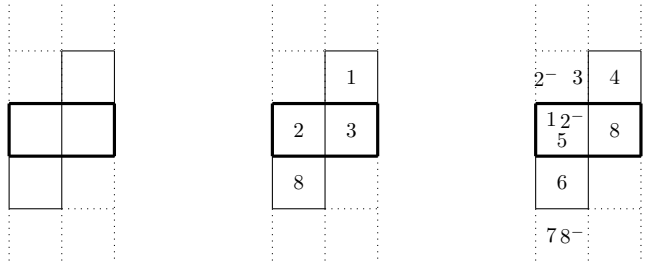
\begin{figure}[h!]
\begin{tikzpicture}[scale=.7]
\begin{scope}
\draw (0,0)--(0,-2);
\draw[very thick] (0,0)--(2,0);
\draw[very thick] (0,0)--(0,-1);
\draw[very thick] (0,-1)--(2,-1);
\draw[very thick](2,0)--(2,-1);
\draw (1,1)--(1,-2);
\draw (2,1)--(2,-1);
\draw (1,1)--(2,1);
\draw (0,-2)--(1,-2);
\draw[dotted](0,2)--(0,0);
\draw[dotted](1,2)--(1,0);
\draw[dotted](2,2)--(2,1);
\draw[dotted](0,1)--(1,1);
\draw[dotted](0,-2)--(0,-3);
\draw[dotted](1,-2)--(1,-3);
\draw[dotted](2,-1)--(2,-3);
\draw[dotted](1,-2)--(2,-2);
\end{scope}

\begin{scope}[xshift=5cm]
\draw (0,0)--(0,-2);
\draw[very thick] (0,0)--(2,0);
\draw[very thick] (0,0)--(0,-1);
\draw[very thick] (0,-1)--(2,-1);
\draw[very thick](2,0)--(2,-1);
\draw (1,1)--(1,-2);
\draw (2,1)--(2,-1);
\draw (1,1)--(2,1);
\draw (0,-2)--(1,-2);
\draw[dotted](0,2)--(0,0);
\draw[dotted](1,2)--(1,0);
\draw[dotted](2,2)--(2,1);
\draw[dotted](0,1)--(1,1);
\draw[dotted](0,-2)--(0,-3);
\draw[dotted](1,-2)--(1,-3);
\draw[dotted](2,-1)--(2,-3);
\draw[dotted](1,-2)--(2,-2);
\node at (.5,-.5)[scale=.8] {2};
\node at (.5,-1.5) [scale=.8]{8};
\node at (1.5,-.5) [scale=.8]{3};
\node at (1.5,.5) [scale=.8]{1};
\end{scope}

\begin{scope}[xshift=10cm]
\draw (0,0)--(0,-2);
\draw[very thick] (0,0)--(2,0);
\draw[very thick] (0,0)--(0,-1);
\draw[very thick] (0,-1)--(2,-1);
\draw[very thick](2,0)--(2,-1);
\draw (1,1)--(1,-2);
\draw (2,1)--(2,-1);
\draw (1,1)--(2,1);
\draw (0,-2)--(1,-2);
\draw[dotted](0,2)--(0,0);
\draw[dotted](1,2)--(1,0);
\draw[dotted](2,2)--(2,1);
\draw[dotted](0,1)--(1,1);
\draw[dotted](0,-2)--(0,-3);
\draw[dotted](1,-2)--(1,-3);
\draw[dotted](2,-1)--(2,-3);
\draw[dotted](1,-2)--(2,-2);
\node at (.2,.5)[scale=.8] {$2^-$};
\node at (.8,.5)[scale=.8] {$3$};
\node at (.3,-.3)[scale=.8] {$1$};
\node at (.7,-.3)[scale=.8] {$2^-$};
\node at (.5,-.7)[scale=.8] {$5$};
\node at (.5,-1.5) [scale=.8]{6};
\node at (.3,-2.5)[scale=.8] {$7$};
\node at (.7,-2.5)[scale=.8] {$8^-$};
\node at (1.5,-.5) [scale=.8]{8};
\node at (1.5,.5) [scale=.8]{4};
\end{scope}
\end{tikzpicture}
\caption{ A skew shape  $r=1, g= d=8, \alpha=(0,1), \beta=(1,0)$ and two fillings.}
\label{fig:fillskewshape}
\end{figure}

\begin{definition}\label{def:admfil}
A  filling of the skew shape is an assignment to the squares in the box of numbers among $\ 1,\dots, g$, each with a positive or negative weight of 1. 
The $i$-weight $\mathbf {w^i(a,b)}$ of a box situated in spot $(a,b)$  is the sum of the weights of the numbers $i'\le i$ that appear on it plus the initial weight of the box.
We  call the filling  an  {\em admissible filling}  if it satisfies
\begin{enumerate}[(a)]
\item Each number appears at most once with weight one (but may appear multiple times with weight $-1$).
\item Each box has at most one occurrence of a given number (including both positive and negative weight).
\item For each choice of an $i, 0\le i\le g$, the $i$-weight of a box  is  either 0 or 1.
\item  The $i$-weight of any given box is greater than or equal to the weight of a box  to its right.
\item   The $i$-weight of any given box is greater than or equal to the weight of a box   below.
\item  The $g$-weight of a box is 1 for the boxes lying above the lower contour of the skew shape and 0 for those lying below that lower contour.
\end{enumerate}
\end{definition}

\begin{figure}[h!]
\begin{tikzpicture}[scale=.55]
\begin{scope}
\fill [color=pink] (0,2)--(0,0)--(1,0)--(1,1)--(2,1)--(2,2)--(0,2);
\draw (0,0)--(0,-2);
\draw[very thick] (0,0)--(2,0);
\draw[very thick] (0,0)--(0,-1);
\draw[very thick] (0,-1)--(2,-1);
\draw[very thick](2,0)--(2,-1);
\draw (1,1)--(1,-2);
\draw (2,1)--(2,-1);
\draw (1,1)--(2,1);
\draw (0,-2)--(1,-2);
\draw[dotted](0,2)--(0,0);
\draw[dotted](1,2)--(1,0);
\draw[dotted](2,2)--(2,1);
\draw[dotted](0,1)--(1,1);
\draw[dotted](0,-2)--(0,-3);
\draw[dotted](1,-2)--(1,-3);
\draw[dotted](2,-1)--(2,-3);
\draw[dotted](1,-2)--(2,-2);
\node at (1,-3.5)[scale=.7]{$w_0$};
\end{scope}

\begin{scope}[xshift=3cm]
\fill [color=pink] (0,2)--(0,-1)--(1,-1)--(1,1)--(2,1)--(2,2)--(0,2);
\draw (0,0)--(0,-2);
\draw[very thick] (0,0)--(2,0);
\draw[very thick] (0,0)--(0,-1);
\draw[very thick] (0,-1)--(2,-1);
\draw[very thick](2,0)--(2,-1);
\draw (1,1)--(1,-2);
\draw (2,1)--(2,-1);
\draw (1,1)--(2,1);
\draw (0,-2)--(1,-2);
\draw[dotted](0,2)--(0,0);
\draw[dotted](1,2)--(1,0);
\draw[dotted](2,2)--(2,1);
\draw[dotted](0,1)--(1,1);
\draw[dotted](0,-2)--(0,-3);
\draw[dotted](1,-2)--(1,-3);
\draw[dotted](2,-1)--(2,-3);
\draw[dotted](1,-2)--(2,-2);
\node at (1,-3.5)[scale=.7]{$w_1$};
\end{scope}

\begin{scope}[xshift=6cm]
\fill [color=pink] (0,2)--(0,1)--(2,1)--(2,2)--(0,2);
\draw (0,0)--(0,-2);
\draw[very thick] (0,0)--(2,0);
\draw[very thick] (0,0)--(0,-1);
\draw[very thick] (0,-1)--(2,-1);
\draw[very thick](2,0)--(2,-1);
\draw (1,1)--(1,-2);
\draw (2,1)--(2,-1);
\draw (1,1)--(2,1);
\draw (0,-2)--(1,-2);
\draw[dotted](0,2)--(0,0);
\draw[dotted](1,2)--(1,0);
\draw[dotted](2,2)--(2,1);
\draw[dotted](0,1)--(1,1);
\draw[dotted](0,-2)--(0,-3);
\draw[dotted](1,-2)--(1,-3);
\draw[dotted](2,-1)--(2,-3);
\draw[dotted](1,-2)--(2,-2);
\node at (1,-3.5)[scale=.7]{$w_2$};
\end{scope}

\begin{scope}[xshift=9cm]
\fill [color=pink] (0,2)--(0,0)--(1,0)--(1,1)--(2,1)--(2,2)--(0,2);
\draw (0,0)--(0,-2);
\draw[very thick] (0,0)--(2,0);
\draw[very thick] (0,0)--(0,-1);
\draw[very thick] (0,-1)--(2,-1);
\draw[very thick](2,0)--(2,-1);
\draw (1,1)--(1,-2);
\draw (2,1)--(2,-1);
\draw (1,1)--(2,1);
\draw (0,-2)--(1,-2);
\draw[dotted](0,2)--(0,0);
\draw[dotted](1,2)--(1,0);
\draw[dotted](2,2)--(2,1);
\draw[dotted](0,1)--(1,1);
\draw[dotted](0,-2)--(0,-3);
\draw[dotted](1,-2)--(1,-3);
\draw[dotted](2,-1)--(2,-3);
\draw[dotted](1,-2)--(2,-2);
\node at (1,-3.5)[scale=.7]{$w_3$};
\end{scope}

\begin{scope}[xshift=12cm]
\fill [color=pink] (0,2)--(0,0)--(2,0)--(2,2)--(0,2);
\draw (0,0)--(0,-2);
\draw[very thick] (0,0)--(2,0);
\draw[very thick] (0,0)--(0,-1);
\draw[very thick] (0,-1)--(2,-1);
\draw[very thick](2,0)--(2,-1);
\draw (1,1)--(1,-2);
\draw (2,1)--(2,-1);
\draw (1,1)--(2,1);
\draw (0,-2)--(1,-2);
\draw[dotted](0,2)--(0,0);
\draw[dotted](1,2)--(1,0);
\draw[dotted](2,2)--(2,1);
\draw[dotted](0,1)--(1,1);
\draw[dotted](0,-2)--(0,-3);
\draw[dotted](1,-2)--(1,-3);
\draw[dotted](2,-1)--(2,-3);
\draw[dotted](1,-2)--(2,-2);
\node at (1,-3.5)[scale=.7]{$w_4$};
\end{scope}

\begin{scope}[xshift=15cm]
\fill [color=pink] (0,2)--(0,-1)--(1,-1)--(1,0)--(2,0)--(2,2)--(0,2);
\draw (0,0)--(0,-2);
\draw[very thick] (0,0)--(2,0);
\draw[very thick] (0,0)--(0,-1);
\draw[very thick] (0,-1)--(2,-1);
\draw[very thick](2,0)--(2,-1);
\draw (1,1)--(1,-2);
\draw (2,1)--(2,-1);
\draw (1,1)--(2,1);
\draw (0,-2)--(1,-2);
\draw[dotted](0,2)--(0,0);
\draw[dotted](1,2)--(1,0);
\draw[dotted](2,2)--(2,1);
\draw[dotted](0,1)--(1,1);
\draw[dotted](0,-2)--(0,-3);
\draw[dotted](1,-2)--(1,-3);
\draw[dotted](2,-1)--(2,-3);
\draw[dotted](1,-2)--(2,-2);
\node at (1,-3.5)[scale=.7]{$w_5$};
\end{scope}

\begin{scope}[xshift=18cm]
\fill [color=pink] (0,2)--(0,-2)--(1,-2)--(1,0)--(2,0)--(2,2)--(0,2);
\draw (0,0)--(0,-2);
\draw[very thick] (0,0)--(2,0);
\draw[very thick] (0,0)--(0,-1);
\draw[very thick] (0,-1)--(2,-1);
\draw[very thick](2,0)--(2,-1);
\draw (1,1)--(1,-2);
\draw (2,1)--(2,-1);
\draw (1,1)--(2,1);
\draw (0,-2)--(1,-2);
\draw[dotted](0,2)--(0,0);
\draw[dotted](1,2)--(1,0);
\draw[dotted](2,2)--(2,1);
\draw[dotted](0,1)--(1,1);
\draw[dotted](0,-2)--(0,-3);
\draw[dotted](1,-2)--(1,-3);
\draw[dotted](2,-1)--(2,-3);
\draw[dotted](1,-2)--(2,-2);
\node at (1,-3.5)[scale=.7]{$w_6$};
\end{scope}

\begin{scope}[xshift=21cm]
\fill [color=pink] (0,2)--(0,-3)--(1,-3)--(1,0)--(2,0)--(2,2)--(0,2);
\draw (0,0)--(0,-2);
\draw[very thick] (0,0)--(2,0);
\draw[very thick] (0,0)--(0,-1);
\draw[very thick] (0,-1)--(2,-1);
\draw[very thick](2,0)--(2,-1);
\draw (1,1)--(1,-2);
\draw (2,1)--(2,-1);
\draw (1,1)--(2,1);
\draw (0,-2)--(1,-2);
\draw[dotted](0,2)--(0,0);
\draw[dotted](1,2)--(1,0);
\draw[dotted](2,2)--(2,1);
\draw[dotted](0,1)--(1,1);
\draw[dotted](0,-2)--(0,-3);
\draw[dotted](1,-2)--(1,-3);
\draw[dotted](2,-1)--(2,-3);
\draw[dotted](1,-2)--(2,-2);
\node at (1,-3.5)[scale=.7]{$w_7$};
\end{scope}

\begin{scope}[xshift=24cm]
\fill [color=pink] (0,2)--(0,-2)--(1,-2)--(1,-1)--(2,-1)--(2,2)--(0,2);
\draw (0,0)--(0,-2);
\draw[very thick] (0,0)--(2,0);
\draw[very thick] (0,0)--(0,-1);
\draw[very thick] (0,-1)--(2,-1);
\draw[very thick](2,0)--(2,-1);
\draw (1,1)--(1,-2);
\draw (2,1)--(2,-1);
\draw (1,1)--(2,1);
\draw (0,-2)--(1,-2);
\draw[dotted](0,2)--(0,0);
\draw[dotted](1,2)--(1,0);
\draw[dotted](2,2)--(2,1);
\draw[dotted](0,1)--(1,1);
\draw[dotted](0,-2)--(0,-3);
\draw[dotted](1,-2)--(1,-3);
\draw[dotted](2,-1)--(2,-3);
\draw[dotted](1,-2)--(2,-2);
\node at (1,-3.5)[scale=.7]{$w_8$};
\end{scope}

\end{tikzpicture}
\caption{ A shading of the boxes whose $i$-weight is 1 (for $ i=0,\dots, 8$) corresponding to  the filling in the right in Figure \ref{fig:fillskewshape}.}
\label{fig:weightskewshape}
\end{figure}
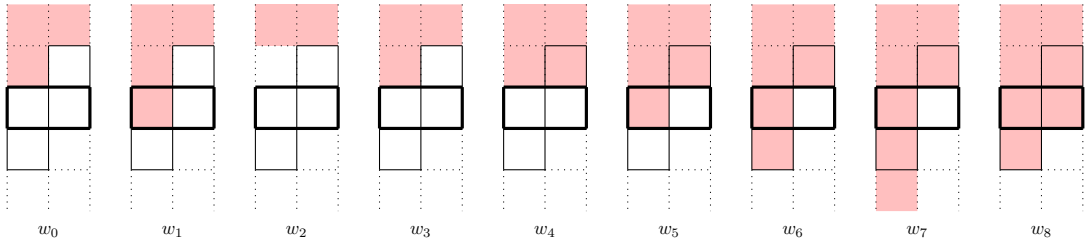

\bigskip
\begin{theorem} \label{teor:components} Let $X$ be a generic chain of $g$ elliptic curves. 
There is a one to one correspondence between the components of the Brill-Noether locus  of limit linear series of degree $d$ 
and dimension $r+1$ with assigned ramification at the points $P^1, Q^g$ and the admissible fillings of the associated skew shape.
\end{theorem}
\begin{remark}
 We are assuming that the filling can go in rows above and below the shape. But due to  the finiteness of $g$ 
 and the requirement that by the end of the process the rows above the lower contour have weight one, the``active'' rows, those that have numbers on them, 
 vary only within finite bounds.
 When we consider the range of variability of a row, we will implicitly assume that it moves within such bounds (any such bounds will do).

\end{remark}
\begin{proof} Fix an allowable filling of a skew shape.
 Define vanishing orders at the nodes inductively  as follows
\[ u^1=(u_0^1,u_1^1, \dots,u_r^1)=(\alpha_0, \alpha_1+1,\dots, \alpha_r+r), \]
\[ u_j ^{i+1} =u_j^i+1-\sum _{ b}w^i(j,b)+\sum _{ b}w^{i-1}(j,b) ,\  i=1,\dots, g-1, \ j=0,\dots, r\]
\[ v_j^i=d-u_j^{i+1}, \ \  i=1,\dots g-1, \ j=0,\dots, r\]
\[v^g=(v_0^g,v_1^g, \dots,v_r^g)=(\beta_0+r, \beta_1+r-1,\dots, \beta_r)\]

Note that the recursive definition of  $u_j^i$ gives
\[ u_j^i=u_j^1+(i-1)-\sum _{ b}w^{i-1}(j,b)+\sum _{ b}w^{0}(j,b)=\alpha_j+j+i-1-\sum _{ b}w^{i-1}(j,b)+\sum _{ b}w^{0}(j,b)=\]
\[ =j+i-1-[\sum _{ b}w^{i-1}(j,b)-\sum _{ b}w^{0}(j,b)-\alpha_j]= j+i-1-  L^{i-1}_j    \]
where $\mathbf{ L^{i-1}_j}$ denotes the lowest row order of a box in column $j$ for which the $(i-1)$-weight is zero, that is 
\begin{equation}\label{eq:uj(i)}    u_j^i=   j+i-1-L^{i-1}_j\text{ where }  w^{i-1}(j, L^{i-1}_j)=0, w^{i-1}(j, L^{i-1}_j-1)=1   \end{equation}

We assign to the filling of the skew shape the component of the set of limit linear series given as $\prod_{i=1}^g G^{r}_d (E_i,P_i,Q_i, u^i,v^i) $.
We need to show the following 
\begin{Claim}\label{nobuit}
\begin{enumerate}[(i)]
\item The sequences $u^i=(u_0^i,u_1^i, \dots,u_r^i)$ satisfy 
\[0\le u_0^i< u_1^i< \dots<u_r^i\le d \text{ (and therefore } 0\ge v_0^i>v_1^i>\dots >v_r^i\ge d) \]
\item Each $ G^{r}_d (E_i,P^i,Q^i, u^i,v^i) $ is non-empty.
\item \[ \sum_{i=1,\dots, g}\dim G^{r}_d (E^i,P^i,Q^i, u^i,v^i) =\rho (g,r,d,\alpha, \beta )\]
\end{enumerate}
\end{Claim}

We first prove (i). From equation (\ref{eq:uj(i)}), this is equivalent to    
\[ j+i-1-L^{i-1}_j  < j+1+i-1-L^{i-1}_{j +1} .\]
which translates to $L^{i-1}_{j +1} \le L^{i-1}_j  $. Then (d) in  Definition \ref{def:admfil} gives the result.

Using Lemma \ref{exfibanul},  condition (ii) above is equivalent to  $u_j^i+v_j^i\le d$ for all $i, j$ and moreover for any given $i$, equality holds for at most one $j$.
From the definition of $v_j^i, i\not= g$, and equation (\ref{eq:uj(i)})
\[ u_j^i+v_j^i=u_j^i+d-u_j^{i+1}= j+i-L^{i-1}_j  +d-[j+i-L^{i}_j  ] =\]
\[= d-1+L^{i}_j   -L^{i-1}_j \]
From Definition  \ref{def:admfil}, (a), $i$ appears at most once with weight one in the filling. 
Hence, $L^{i}_j  \le L^{i-1}_j  +1$ with equality (once $i$ is fixed) for at most one $j$.

For $i=g$, using the definition of $v_j^g$ and equation (\ref{eq:uj(i)}), 
\[ u_j^g+v_j^g=j+g-1-L^{g-1}_j  +\beta_j+r-j=g-1+r+\beta_j-L^{g-1}_j\]
From definition \ref{def:admfil} (d), $L^g_j=g-d+r+\beta_j$ and therefore $d=g+r+\beta_j-L^g_j$. Therefore, condition (ii) for $i=g$ is equivalent to 
\[ g-1+r+\beta_j-L^{g-1}_j\le g+r+\beta_j-L^g_j\]
with equality ocurring for at most one value of $j$. This is equivalent to $L^g_j\le L^{g-1}_j+1$ with equality for at most one value of $j$.
This follows again from  Definition  \ref{def:admfil} (a).

By lemma \ref{exfibanul}  , the sum of dimensions appearing in condition (iii) can be computed as  
\[ \sum_{i=1,\dots, g}[(r+1)d-r-\sum_{j=0,\dots, r} (u_j^i+v_j^i)] \]
Using the definitions of the $u_j^i, v_j^i$, the expression becomes
\[ g(r+1)d-gr- \sum_{j=0,\dots, r}   [\sum_{i=2,\dots, g}(u_j^i+  \sum_{i=1,\dots, g-1}(d-u_j^{i+1})+[\alpha _j+j+\beta_j+(r-j)]]=   \] 
\[ =g(r+1)d-gr-(g-1)(r+1)d -r(r+1)-\sum_{j=0,\dots, r} [\alpha _j+\beta_j]= \rho (g,r,d,\alpha, \beta)   \] 
as needed.)

\medskip

Let us now prove the converse. Given a component of the space of limit linear series on a generic chain of elliptic curves, it can be expressed as 
$\prod_{i=1}^g G^{r}_d (E^i,P^i,Q^i, u^i,v^i )$ where $v_j^{i+1}=d-u_j^i, u_j^1=\alpha_j+j,  v_j^g=\beta _j+r-j$ and 
\[ \sum_{i=1,\dots, g}\dim G^{r}_d (E^i,P^i,Q^i, u^i,v^i) =\rho (g,r,d,\alpha, \beta )\]
We will assign to this component a filling of the skew shape associated to $g, r, d, \alpha, \beta$ by successively adding to the box the indices 1, then 2, then 3 and so on till $g$
with positive or negative weights as follows:
Assume that the indices up to $i-1$ have already been placed. In particular, the concept of $\bar i$-weight makes sense for $\bar i\le i-1$.
Assume that the filling with the first $i-1$ indices satisfies conditions (a)-(e) in Definition \ref{def:admfil} for $\bar i\le i-1$.
In particular, from conditions (c) and (e), for each column $j$ and for every $\bar i \le i-1$,  there is a first row order  $L^{\bar i}_j$ where the  box in column $j$
 has  $\bar i$-weight zero ($  w_{\bar i}(j, L^{\bar i}_j)=0, w_{\bar i}(j, L^{\bar i}_j-1)=1$).

Write $d-u_j^i-v_j^i=k_j^i$.  From Lemma \ref{exfibanul} , $k_j^i\ge 0$. If $k_j^i=0$, insert the index $i$ with positive weight one  in column $j$ and row $L^{ i-1}_j$. 
  If $k_j^i=1$, the index $i$ will not appear in column $j$. 
  If $k_j^i>1$, insert the index $i$ with negative weight   in column $j$ and row $L^{ i-1}_j-1,  L^{ i-1}_j-2, \dots , L^{ i-1}_j-(k_j^i-1)$. 
  Note that this implies that 
  \begin{equation}\label{eq:L(i,j)} L^{ i}_j=L^{ i-1}_j+1-(d-u_j^i-v_j^i)\end{equation}
From Lemma \ref{exfibanul},  conditions (a) in Definition \ref{def:admfil} is satisfied for $\bar i\le i$ and by construction (b), (c) and (e) also hold.
 We need to check that condition (d) is also satisfied.

We start by showing that with our definitions,  
\begin{equation}\label{eq:u}u_j^{\bar i}=   j+\bar i-1-L^{ i-1}_j, \bar i\le i \end{equation}
This is true for $i=1$ as by assumption $u_j^1=\alpha_j+j$ while $L^{ 0}_j=-\alpha _j$.

Let us assume then that it is true up to $i$ and we show it is true for $i+1$.
By assumption, $u_j^{i+1}=d-v_j^i$. From equation (\ref{eq:L(i,j)}),   $d-v_j^i=L^{ i-1}_j-L^{ i}_j,+1+u_j^i$.
Then (\ref{eq:u}) for $i+1$ follows from the corresponding result for $i$.

Now condition $u_j^i<u_{j+1}^i$ can be written as 
\[ j+ i-1-L^{ i-1}_j<  j+1+ i-1-L^{ i-1}_{j +1}\]
which is equivalent to 
\[ L^{ i-1}_{j +1}\le L^{ i-1}_{j }\]
This is condition (d) in Definition \ref{def:admfil}.

It only remains to show that the $g$-weight of a box is 1 for the boxes lying above the lower contour of the skew shape and 0 for those lying below that lower contour.
With our notations, this condition is written as 
\[ L^g_j=g-d+r+\beta_j, \ j=0,\dots, r\]
We can compute the value of $L^g_j$ using repeatedly equation (\ref{eq:L(i,j)}):
\[ L^g_j= L^{g-1}_j+1+u_j^g+v_j^g-d=\dots=L^1_j+g-1+\sum_{i=1,\dots,g}[u_j^i+v_j^i-d]=\]
\[= L^1_j+g-1+u_j^1+\sum_{i=2,\dots,g}u_j^i+\sum_{i=1,\dots,g-1}(d-u_j^{i+1})+v_j^g-d= L(1,j)+g-1+u_j^1+v_j^g-d\]
From our definition of filling of a skew shape, $L^1_j=-\alpha_j$. By assumption, $u_j^1=\alpha_j+j, v_j^g=d-\beta_j+r-j$. 
Making these substitutions in the equation above, we obtain 
\[ L^g_j=g-d+r+\beta_j\]
as needed. This concludes the proof of the Theorem.
\end{proof}

\section{Intersections}

We describe the intersection of components of the spaces of limit linear series in terms of the corresponding  fillings (see Definition \ref{def:admfil} and Theorem \ref{teor:components}) of a fixed skew shape.

Geometrically, the condition for a non-empty intersection is best described in terms of the  shading (such as the one in Figure \ref{fig:weightskewshape})
of the admissible fillings of the components we are intersecting. 
If the intersection is non-empty,  at each stage $i$ the shading of each skew shape  has at most one more box that the shading from all skew shapes at stage $i-1$.
 Moreover, the codimension of the intersection is the total discrepancy in the number of boxes that appear  in the shading of the different fillings.
 
 Note that the condition that every index appears at most once with positive weight one in a given admissible filling
  guarantees that the component corresponding to a skew shape intersects itself (and the co-dimension of the intersection is zero). 
  
  For a filling $A(k)$ of a skew shape, we will denote   by $L^i_j(k)$ the highest row order of a box in column $j$ for which the $(i)$-weight is one for  the filling  $A(k)$.

\begin{theorem}\label{intersection} 
The intersection of components $A(1),\dots , A(t)$ corresponding to  different fillings of a skew shape is the product over the elliptic curves $C^i$ of the loci of linear series of degree
$d$ with vanishing $\bar u_j^i$ at $P^i$  and vanishing $\bar v_j^i$ at $Q^i$  given by 
\[ \bar u_j^i=j+i-1-\min_{k\in\{1,\dots, t\} }L_j^{i-1}( k),  \ \   \bar v_j^i=d-j-i+\max_{k\in\{1,\dots, t\} }L_j^{i}( k),\]
In particular, the intersection is non-empty if and only if 
\[   \max_{k\in\{1,\dots, t\}} L_j^{i}( k)\le \min_{k\in\{1,\dots, t\} }L_j^{i-1}( k)+1\] 
and  for each $i$  equality occurs for at most one $j$.
  The codimension of the intersection is given as 
  \[ \sum _{i=1}^{g-1}\sum_{j=0}^r[\max_{k\in\{1,\dots, t\} }L_j^{i}( k)- \min_{k\in\{1,\dots, t\} }L_j^{i}( k), \]
\end{theorem}

\begin{proof} The space of limit linear series is a product over the elliptic curves of a space of linear series on each curve with given vanishing at the points $P^i, Q^i$.
The intersection of  components corresponding to  different fillings of a skew shape can then be obtained as the product over the elliptic curves
 of the intersections of the spaces of linear series.
The condition for  components to intersect can be formulated curve-wise:
on each elliptic component there is an $r+1$-dimensional space of sections satisfying simultaneously the vanishing conditions for each of the fillings. 
 If the vanishing at $P^i, Q^i$ imposed by the $k^{th}$ skewshape  is given by $(u^i_0(k),\dots , u^i_r(k)),  (v^i_0(k),\dots , v^i_r(k)) $,
  then the intersection on the component $C^i$ is the space of linear series with vanishing 
 \[  (\max_{k\in\{1,\dots, t\} } \{u^i_0(k)\},\dots ,\max_{k\in\{1,\dots, t\} } \{u^i_r(k) \}),  (\max_{k\in\{1,\dots, t\} } \{v^i_0(k)\},\dots , \max_{k\in\{1,\dots, t\} } \{v^i_r(k) \})\]

 Recall that, each filling of a skew shape  determines, the minimum orders of vanishing of sections at $P^i, Q^i$  on the elliptic component $C^i$.
From equation \ref{eq:uj(i)}, we can write
$u^i_j(k) = j+i-1-L_j^{i-1}(k)$ while $v^i_j(k)=d-u^{i+1}_j(k).$
Therefore, 
\[  \bar u^i_j(k)=  \max_{k\in\{1,\dots, t\} } \{u^i_j(k)\} =j+i-1-\min_{k\in\{1,\dots, t\} } \{L_j^{i-1}(k)\}, \]
\[ \bar v^i_j(k) =\max_{k\in\{1,\dots, t\} } \{v^i_j(k)\} =d-j-i+\max_{k\in\{1,\dots, t\} } \{L_j^i(k)\}, \]
Using Lemma \ref{exfibanul},  the condition for the intersection to be non-empty is equivalent to  
\[   \bar u_j^i+\bar v_j^i\le d,   \text{ for all } i, j\]
 where for a given $i$, at most one of the inequalities is an equality.
 Combining this equation with the description of $ \bar u^i_j(k),  \bar v^i_j(k)$ in the previous equations, this gives the non-emptiness condition in the statement of the Theorem.
 
From Lemma \ref{exfibanul} , the dimension of the intersection is $(r+1)d-r-\sum_{j=0,\dots, r} (\bar u_j^i+\bar v_j^i)$.
This can be translated as 
  \[ \sum _{i=1}^g [ 1+\sum_j\min_{k\in\{1,\dots, t\} }L_j^{i-1}(k)\-\max_{k\in\{1,\dots, t\} }L_j^{i}(k)\  =  \]  
  \[ =g-\sum_{j=0}^r[\max_{k\in\{1,\dots, t\} }L_j^{g}(k) -\min_{k\in\{1,\dots, t\} }L_j^{0}(k)]   -\sum _{i=1}^{g-1}\sum_{j=0}^r[\max_{k\in\{1,\dots, t\} }L_j^{i}(k)- \min_{k\in\{1,\dots, t\} }L_j^{i}(k)] \]
Note now that the dimension $\rho$ of every component is $g$ minus the number of boxes in the shape which can be written as $g-(\sum_jL^g_j-\sum_jL^0_j)$ 
where the $L_j^{i}(k)$ can refer to any filling of the skew shape as they all start empty and end full after the $g^{th}$ stage.
This gives the stated result for the codimension of the intersection.
\end{proof}


\end{document}